\documentclass[11pt,reqno]{amsart}
\usepackage{color}
\usepackage[english]{babel}
\usepackage{amsfonts}
\usepackage{amsmath}
\usepackage{amsthm}
\usepackage{amssymb}
\usepackage[misc]{ifsym}
\usepackage{bbm}
\usepackage{mathrsfs}
\usepackage{esint}
\usepackage{faktor}
\usepackage[utf8]{inputenc}
\usepackage{afterpage}
\usepackage[left=2.9cm,right=2.9cm,top=2.8cm,bottom=2.8cm]{geometry}
\usepackage{graphicx}
\usepackage[pagewise,displaymath,mathlines]{lineno}
\usepackage{enumitem}
\usepackage[dvipsnames]{xcolor}
\usepackage[colorlinks=true,urlcolor=blue,citecolor=blue,linkcolor=blue,linktocpage,pdfpagelabels,bookmarksnumbered,bookmarksopen]{hyperref}

\newcommand{\N}{\mathbb{N}}

\newcommand{\R}{\mathbb{R}}

\newcommand{\D}{\mathcal{D}}

\newcommand{\dx}{\, {\rm d} x}

\newcommand{\drho}{\, {\rm d} \rho}
\newcommand{\eps}{\varepsilon}
\newcommand{\loc}{{\rm loc}}
\renewcommand{\phi}{\varphi}

\newtheorem{lemma}{Lemma}[section]
\newtheorem{thm}[lemma]{Theorem}
\newtheorem{prop}[lemma]{Proposition}

\theoremstyle{definition}

\newtheorem{rmk}[lemma]{Remark}

\numberwithin{equation}{section}

\DeclareMathOperator*{\essinf}{ess \, inf}
\DeclareMathOperator*{\supp}{supp}
\DeclareMathOperator*{\dist}{dist}

\begin{document}

\title[A note on gradient estimates for $p$-Laplacian equations]{A note on gradient estimates for $p$-Laplacian equations}
\author[U. Guarnotta]{Umberto Guarnotta}
\address[U. Guarnotta]{Dipartimento di Matematica e Informatica, Universit\`a degli Studi di Catania, Viale A. Doria 6, 95125 Catania, Italy}
\email{umberto.guarnotta@studium.unict.it}
\author[S.A. Marano]{Salvatore A. Marano}
\address[S.A. Marano]{Dipartimento di Matematica e Informatica, Universit\`a degli Studi di Catania, Viale A. Doria 6, 95125 Catania, Italy}
\email{marano@dmi.unict.it}

\maketitle

\begin{abstract}
The aim of this short paper is to show that some assumptions in \cite{GMM} can be relaxed and even dropped when looking for weak solutions instead of strong ones. This improvement is a consequence of two results concerning gradient terms: an $L^\infty$ estimate, which exploits nonlinear potential theory, and a compactness result, based on the classical Riesz-Fréchet-Kolmogorov theorem.
\end{abstract}

{
\let\thefootnote\relax
\footnote{\Letter \quad Corresponding author: Umberto Guarnotta (umberto.guarnotta@studium.unict.it).}
}
\setcounter{footnote}{0}

\section{Introduction}

In this brief note, whose starting point is \cite{GMM}, we consider the problem
\begin{equation}
\label{prob}
\tag{${\rm P}$}
\left\{
\begin{alignedat}{2}
-\Delta_p u &= f(x,u,v,\nabla u, \nabla v) \; \; &&\mbox{in} \; \; \R^N, \\
-\Delta_q v &= g(x,u,v,\nabla u, \nabla v) \; \; &&\mbox{in} \; \; \R^N, \\
u,v &> 0 \; \; &&\mbox{in} \; \; \R^N,
\end{alignedat}
\right.
\end{equation}
where $N\geq 2$, $1<p,q<N$, $\Delta_r z := {\rm div}(|\nabla z|^{r-2}\nabla z)$ denotes the $r$-Laplacian of $z$ for $1<r<+\infty$, while $f,g:\R^N\times(0,+\infty)^2\times\R^{2N}\to(0,+\infty)$ are Carathéodory functions satisfying the following hypotheses.

\begin{enumerate}[label={$ \underline{\rm H_1(f)} $},ref={$ \rm H_1(f) $}]
\item \label{hypf}
There exist $\alpha_1\in (-1,0]$, $\beta_1,\delta_1\in [0, q-1)$, $\gamma_1\in [0,p-1)$, $ m_1, \hat{m}_1 > 0 $, and a measurable $a_1:\R^N\to(0,+\infty)$ such that
\begin{equation*}
m_1 a_1(x) s_1^{\alpha_1} s_2^{\beta_1} \leq f(x,s_1,s_2,{\bf t}_1,{\bf t}_2) \leq \hat{m}_1 a_1(x) \left( s_1^{\alpha_1} s_2^{\beta_1} + |{\bf t}_1|^{\gamma_1} + |{\bf t}_2|^{\delta_1} \right)
\end{equation*}
in $\R^N \times (0,+\infty)^2 \times\R^{2N}$. Moreover, $\displaystyle{\essinf_{B_\rho} a_1 > 0}$ for all $\rho>0$.
\end{enumerate}
\begin{enumerate}[label={$ \underline{\rm H_1(g)} $},ref={$ \rm H_1(g) $}]
\item \label{hypg}
There exist $\beta_2\in (-1,0]$, $\alpha_2, \gamma_2\in [0,p-1)$, $\delta_2 \in [0,q-1)$, $ m_2, \hat{m}_2 > 0 $, and a measurable $a_2:\R^N\to(0,+\infty)$ such that
\begin{equation*}
m_2 a_2(x) s_1^{\alpha_2} s_2^{\beta_2} \leq g(x,s_1,s_2,{\bf t}_1,{\bf t}_2) \leq \hat{m}_2 a_2(x) \left( s_1^{\alpha_2} s_2^{\beta_2} + |{\bf t}_1|^{\gamma_2} + |{\bf t}_2|^{\delta_2} \right)
\end{equation*}
in $ \R^N \times (0,+\infty)^2 \times\R^{2N} $. Moreover, $\displaystyle{\essinf_{B_\rho} a_2 > 0}$ for all $\rho>0$.
\end{enumerate}
\begin{enumerate}[label={$ \underline{\rm H_1(a)} $},ref={$ \rm H_1(a) $}]
\item \label{hypa}
There exist $ \zeta_1,\zeta_2 \in (N,+\infty] $ such that $ a_i \in L^1(\R^N) \cap L^{\zeta_i}(\R^N) $, $ i=1,2 $, where
\begin{equation*}
\frac{1}{\zeta_1} < 1 - \frac{p}{p^*} - \theta_1, \quad \frac{1}{\zeta_2} < 1 - \frac{q}{q^*} - \theta_2,
\end{equation*}
with
\begin{equation*}
\begin{split}
\theta_1 :=\max\left\{\frac{\beta_1}{q^*},\frac{\gamma_1}{p},\frac{\delta_1}{q}\right\} < 1 - \frac{p}{p^*}, \quad\theta_2 :=\max\left\{\frac{\alpha_2}{p^*},\frac{\gamma_2}{p},\frac{\delta_2}{q}\right\}< 1-
\frac{q}{q^*}.
\end{split}
\end{equation*}
\end{enumerate}
\begin{enumerate}[label={$ \underline{\rm H_2} $},ref={$ \rm H_2 $}]
\item \label{hypeta}
If $ \eta_1 := \max\{\beta_1,\delta_1\} $ and $ \eta_2 := \max\{\alpha_2,\gamma_2\} $ then
\begin{equation*}
\eta_1 \eta_2 < (p-1-\gamma_1)(q-1-\delta_2).
\end{equation*}
\end{enumerate}
In the sequel, by \hypertarget{H1}{${\rm H_1}$} we mean the set of hypotheses \ref{hypf}, \ref{hypg}, and \ref{hypa}.

Unlike \cite{GMM}, we restrict our attention to weak solutions instead of strong ones. This allows us to weaken several assumptions, in particular:
\begin{itemize}
\item $p,q>2-\frac{1}{N}$ is relaxed to $p,q>1$;
\item hypothesis ${\rm H_3}$ (cf.~\cite[p.~743]{GMM}), ensuring a high local summability of reactions, is dropped;
\item no local summability for $a_1,a_2$ is required (cf.~\ref{hypf}--\ref{hypg}).
\end{itemize}
Let us briefly comment these improvements, focusing our attention on the first equation of \eqref{prob}, since our arguments are scalar (i.e., they do not depend on the system structure). The lower bound on $p$ was used to prove \cite[Lemma 2.1]{GMM} and, jointly with ${\rm H_3}$, to guarantee the strong convergence of $\{|\nabla u_n|^{p-2}\nabla u_n\}$ in $L^{p'}_\loc(\R^N)$, being $\{u_n\}$ a sequence of solutions (precisely, their first components) to problems approximating \eqref{prob} (see \cite[formula (4.5)]{GMM}). On the other hand, in the assumption $a_1\in L^{s_p}_{\rm loc}(\R^N)$, the number $s_p$ was supposed to be greater than $p'N$, which ensures the local $C^{1,\alpha}$ regularity in \cite[Lemma 3.1]{GMM} by \cite{DB}. Actually, the same result can be obtained by requiring merely $s_p>N$, according to \cite{L}, so that one can take $s_p:=\zeta_1$ with no additional assumptions, where $\zeta_1$ stems from \ref{hypa}. Another consequence of using \cite{L} instead of \cite{DB} is that ${\rm H_3'}$ in \cite[Remark 4.4]{GMM} can be relaxed to
\begin{equation*}
\frac{1}{s_p} +\max \left\{ \frac{\gamma_1}{p},\frac{\delta_1}{q} \right\} < \frac{1}{N}\, ,\quad
\frac{1}{s_q} + \max \left\{ \frac{\gamma_2}{p},\frac{\delta_2}{q} \right\} < \frac{1}{N}\, .
\end{equation*}

Convergence of gradient terms comes into play whenever a second-order differential problem needs to be approximated: this can occur because of lack of ellipticity (or uniform ellipticity) of the principal part and/or presence of non-smooth reaction terms; see, e.g., \cite[Theorem 3.3]{GM}. An approximation procedure is necessary also in the context of singular problems, that is, problems whose reaction term blows up when the solution approaches to zero, as \eqref{prob}; for an account on this topic, vide \cite{GLM1,GLM2}.

Here, we proceed as follows. Lemma 2.1 of \cite{GMM} is restated in a new, general fashion and its proof is given patterned after the one in \cite{GMM}; see Lemma \ref{mingionecor}. Next, we prove a compactness result (Lemma \ref{convgrad}) for gradient terms, which is self-contained (unlike the alternative proofs mentioned in Remark \ref{alternative}) and relies on the basic Riesz-Fréchet-Kolmogorov $L^p$-compactness criterion. Finally, it is shown (in Theorem \ref{weaksol}) how to modify the proof of \cite[Lemma 4.1]{GMM} to get a weak solution under \hyperlink{H1}{${\rm H_1}$}--\ref{hypeta}, besides commenting the unavailability of \cite[Lemma 4.3]{GMM}, pertaining strong solutions, in this context (see Remark \ref{strongsol}).

\subsection*{Notations}
Hereafter $\Omega$ is a bounded domain of $\R^N$, $N\geq 2$, and $p\in(1,+\infty)$. We set $p':=\frac{p}{p-1}$ and, provided $p<N$, $p^*:=\frac{Np}{N-p}$. If $p\geq N$ then $p^*:=\infty$ and $(p^*)':=1$. We write $\dist(A,B)$ for the distance between the sets $A,B\subseteq \R^N$. The symbol $B_R(x)$ indicates the (open) ball of center $x\in\R^N$ and radius $R>0$, while $\overline{B}_R(x)$ stands for the closure of $B_R(x)$. By $B_R(x) \Subset \Omega$ we mean $\overline{B}_R(x) \subseteq \Omega$. The center of any ball will be omitted when it is irrelevant. \\
Given $f\in L^1_\loc(\R^N)$, a distributional solution to
\begin{equation}
\label{baseprob}
-\Delta_p u = f(x) \quad \mbox{in}\;\; \R^N
\end{equation}
is a function $u\in W^{1,p}_\loc(\R^N)$ such that
\begin{equation}
\label{weakform}
\int_\Omega |\nabla u|^{p-2}\nabla u \cdot \nabla \phi \dx = \int_\Omega f\phi \dx \quad \forall \phi\in C^\infty_c(\R^N).
\end{equation}
If $f\in L^{(p^*)'}(\R^N)$, by weak solution to \eqref{baseprob} we mean a function $u\in \D^{1,p}_0(\R^N)$ satisfying \eqref{weakform} for all $\phi\in\D^{1,p}_0(\R^N)$. Analogous definitions hold when $\Omega$ replaces $\R^N$ or $f$ depends on $u,\nabla u$. \\
The number $C>0$ represents a suitable constant, which may change its value at each passage. For further details, we address the reader to \cite[Section 2]{GMM}.

\section{Main results}

For any $ f \in L^2_\loc(\Omega) $ we define the nonlinear potential
\begin{equation*}
\begin{split}
&P_f(x,R) := \int_0^R \left( \frac{|f|^2(B_\rho(x))}{\rho^{N-2}} \right)^{\frac{1}{2}} \frac{\!\drho}{\rho},  \quad \mbox{with} \; \; |f|^2(B_\rho(x)) := \|f\|_{L^2(B_\rho(x))}^2.
\end{split}
\end{equation*}

We recall the following result, provided in \cite{DM}.
\begin{prop}
\label{mingionethm}
Let $ u \in W^{1,p}_\loc(\Omega) $ be a distributional solution to 
\begin{equation}
\label{mingioneprob}
-\Delta_p u = f(x) \quad \mbox{in} \;\; \Omega,
\end{equation}
with $ f \in L^r_{\rm loc}(\Omega) $, $ r :=\max\{2,(p^*)'\} $. Then there exists $ C = C(N,p) > 0 $ such that
\begin{equation*}
\begin{split}
\|\nabla u\|_{L^\infty(B_R)} &\leq C \left[\left( \frac{1}{|B_{2R}|} \int_{B_{2R}} |\nabla u|^p \dx \right)^{\frac{1}{p}} + \|P_f(\cdot,2R)\|_{L^\infty(B_{2R})}^{\frac{1}{p-1}} \right]
\end{split}
\end{equation*}
for any $ B_{2R} \Subset \Omega $.
\end{prop}
\begin{rmk}
\label{veryweak}
As observed in \cite[p. 1363]{DM}, the condition $ r \geq (p^*)' $ is not used to prove the result, but it guarantees that $ u $ is a weak solution, and not merely a \textit{very weak solution}; in the latter case, an approximation procedure yields the existence of a very weak solution $ u \in W^{1,p-1}(\Omega) $ of \eqref{mingioneprob}. For a thorough treatment on approximable solutions, see \cite{CM-NA}.
\end{rmk}

\begin{lemma}
\label{mingionecor}
Let $ u \in \mathcal{D}^{1,p}_0(\R^N) $ be a distributional solution to
\begin{equation*}
-\Delta_p u = f(x) \quad \mbox{in} \;\; \R^N,
\end{equation*}
with $f\in L^r(\R^N)$, $r > N$. Then $ \nabla u \in L^\infty(\R^N) $. More precisely, there exists $ C = C(N,p) > 0 $ such that
\begin{equation*}
\|\nabla u\|_{L^\infty(\R^N)}^{p-1} \leq C \left(\|\nabla u\|_{L^p(\R^N)}^{p-1} + \|f\|_{L^r(\R^N)}\right).
\end{equation*}
\end{lemma}
\begin{proof}
Pick any $ x \in \R^N $. By Proposition \ref{mingionethm} and H\"older's inequality (with exponents $\frac{r}{2}$ and $\frac{r}{r-2}$), after observing that $ r > N \geq \max\{2,(p^*)'\} $, we get
\begin{equation*}
\begin{split}
|\nabla u(x)|^{p-1} &\leq \|\nabla u\|_{L^\infty(B_1(x))}^{p-1} \\
&\leq C \left[ \left( \frac{1}{|B_2(x)|} \int_{B_2(x)} |\nabla u|^p \dx \right)^{\frac{1}{p'}} + \|P_f(\cdot,2)\|_{L^\infty(B_2(x))} \right] \\
&\leq C \left[ \|\nabla u\|_{L^p(\R^N)}^{p-1} + \sup_{y \in B_2(x)} \, \int_0^2 \rho^{-\frac{N}{2}} \|f\|_{L^2(B_\rho(y))} \drho \right] \\
&\leq C \left[ \|\nabla u\|_{L^p(\R^N)}^{p-1} + \|f\|_{L^r(\R^N)} \int_0^2 \rho^{-\frac{N}{r}} \drho \right] \\
&\leq C \left(\|\nabla u\|_{L^p(\R^N)}^{p-1} + \|f\|_{L^r(\R^N)} \right).
\end{split}
\end{equation*}
Taking the supremum in $ x \in \R^N $ on the left yields the conclusion.
\end{proof}

For every $ u \in W^{1,p}_{\rm loc}(\Omega) $, $ x \in B_R \Subset \Omega $, and $ h \in \R^N $ such that $ |h| < \dist(B_R, \partial \Omega) $, we set
\[
u_h(x) := u(x+h), \quad \delta_h u := u_h-u.
\]
Analogous definitions hold for vector-valued functions.
\begin{lemma}
\label{convgrad}
Let $ \{u_n\} \subseteq W^{1,p}_{\rm loc}(\Omega) $ and $ \{f_n\} \subseteq L^{r'}_{\rm loc}(\Omega) $, $ r \in (1,p^*) $, be such that $ u_n $ is a distributional solution to
\begin{equation*}
-\Delta_p u_n = f_n(x) \quad \mbox{\rm in} \; \; \Omega
\end{equation*}
for all $ n \in \N $. Suppose that:
\begin{enumerate}[label={$ {\rm (K_{\arabic*})} $},itemsep=5px]
\item \label{gradbound} \makebox[\textwidth][c]{$ \{\nabla u_n\} \; \; \mbox{is bounded in} \; \; L^p_{\rm loc}(\Omega); $}
\item \label{reactbound} \makebox[\textwidth][c]{$ \{f_n\} \; \; \mbox{is bounded in} \; \; L^{r'}_{\rm loc}(\Omega); $}
\item \label{conv} \makebox[\textwidth][c]{$ u_n \to u \;\; \mbox{in} \; \; L^p_{\rm loc}(\Omega) \cap L^r_{\rm loc}(\Omega). $}
\end{enumerate}
Then $ \{\nabla u_n\} $ admits a strongly convergent subsequence in $ L^p_{\rm loc}(\Omega) $.
\end{lemma}
\begin{proof}
Fix $ R > 0 $ such that $ B_R \Subset \Omega $. A density argument produces
\begin{equation}
\label{comptest}
\int_{B_R} |\nabla u_n|^{p-2} \nabla u_n \cdot \nabla \phi \dx = \int_{B_R} f_n \phi \dx
\end{equation}
for any $ n \in \N $ and $ \phi \in W^{1,p}_0(B_R) $. Now pick $ t,s > 0 $ such that $ B_t \Subset B_s \Subset B_R $ and $ \eta \in C^\infty_c(B_s) $ such that $ 0 \leq \eta \leq 1 $, $ \eta \equiv 1 $ on $ B_t $, and $ |\nabla \eta| \leq \frac{C}{s-t} $ for some $C>0$. If $ V_n := |\nabla u_n|^{p-2} \nabla u_n $ then using \eqref{comptest} with $ \phi := \eta^2 \delta_h u_n $, where $ |h| < R-s $, gives
\begin{equation}
\label{comptest1}
\int_{B_R} \eta^2 \, V_n \cdot \delta_h(\nabla u_n) \dx + 2 \int_{B_R} \eta \, \delta_h u_n \, V_n \cdot \nabla \eta \dx = \int_{B_R} f_n \phi \dx.
\end{equation}
Next, exploit \eqref{comptest} with $ \phi_{-h} $, perform the change of variable $ x \mapsto x+h $ on the left-hand side, and recall that $ B_{s+|h|} \subseteq B_R $, to achieve
\begin{equation}
\label{comptest2}
\int_{B_R} \eta^2 \, (V_n)_h \cdot \delta_h(\nabla u_n) \dx + 2 \int_{B_R} \eta \, \delta_h u_n \, (V_n)_h \cdot \nabla \eta \dx = \int_{B_R} f_n \phi_{-h} \dx.
\end{equation}
So, subtracting \eqref{comptest1} from \eqref{comptest2} yields
\begin{equation*}
\int_{B_R} \eta^2 \, \delta_h V_n \cdot \delta_h(\nabla u_n) \dx + 2 \int_{B_R} \eta \, \delta_h u_n \, \delta_h V_n \cdot \nabla \eta \dx = \int_{B_R} f_n \delta_{-h} \phi \dx.
\end{equation*}
Since $\supp \eta \subseteq B_s$, this entails
\begin{equation}
\label{comptest3}
\begin{split}
\int_{B_t} \delta_h V_n \cdot \delta_h(\nabla u_n) \dx &\leq \int_{B_R} \eta^2 \, \delta_h V_n \cdot \delta_h(\nabla u_n) \dx \\
&\leq 2 \int_{B_R}  |\delta_h u_n| |\delta_h V_n| |\nabla \eta| \dx + \int_{B_R} |f_n| |\delta_{-h} \phi| \dx \\
&\leq \frac{C}{s-t} \, \|\delta_h u_n\|_{L^p(B_s)} \|\delta_h V_n\|_{L^{p'}(B_s)} + \|f_n\|_{L^{r'}(B_s)} \|\delta_{-h} \phi\|_{L^r(B_s)} \\
&\leq \frac{C}{s-t} \, \|\delta_h u_n\|_{L^p(B_R)} \left(\|(V_n)_h\|_{L^{p'}(B_s)} + \|V_n\|_{L^{p'}(B_s)}\right) \\
&\quad + \|f_n\|_{L^{r'}(B_R)} \left(\|\phi_{-h}\|_{L^r(B_s)} + \|\phi\|_{L^r(B_s)}\right) \\
&\leq \frac{2C}{s-t} \, \|\delta_h u_n\|_{L^p(B_R)} \|V_n\|_{L^{p'}(B_R)} + 2 \, \|f_n\|_{L^{r'}(B_R)} \|\delta_h u_n\|_{L^r(B_R)} \\
& \leq C \left(\|\delta_h u_n\|_{L^p(B_R)} \|\nabla u_n\|_{L^p(B_R)}^{p-1} + \|f_n\|_{L^{r'}(B_R)} \|\delta_h u_n\|_{L^r(B_R)}\right),
\end{split}
\end{equation}
where H\"older's inequality has been used twice while $ C = C(N,t,s) > 0 $. Notice that, thanks to \ref{gradbound}--\ref{conv} and \cite[Exercise 4.34]{B}, the last term of \eqref{comptest3} vanishes as $ h \to 0^+ $ uniformly in $ n $. 
Let us now distinguish two cases, namely $ p \geq 2 $ and $ p \in (1,2) $. \\
\textbf{Case 1.} If $ p \geq 2 $ then
\begin{equation}
\label{compsx1}
\begin{split}
\int_{B_t} \delta_h V_n \cdot \delta_h(\nabla u_n) \dx &= \int_{B_t} (|\nabla (u_n)_h|^{p-2} \nabla (u_n)_h - |\nabla u_n|^{p-2} \nabla u_n) \cdot (\nabla (u_n)_h - \nabla u_n) \dx \\
&\geq c \, \|(\nabla u_n)_h - \nabla u_n\|_{L^p(B_t)}^p = c \, \|\delta_h (\nabla u_n)\|_{L^p(B_t)}^p,
\end{split}
\end{equation}
with $c>0$ small enough; cf. \cite[Chapter 12, inequality (I)]{Lind}. By \eqref{comptest3}--\eqref{compsx1} we thus obtain $ \delta_h (\nabla u_n) \to 0 $ in $ L^p(B_t) $ as $ h \to 0^+ $  uniformly in $ n $, and the Riesz-Fréchet-Kolmogorov yields the conclusion, because $t>0$ was arbitrary. \\
\textbf{Case 2.} For $ p \in (1,2) $ one has  (see \cite[Chapter 12, inequality (VII)]{Lind})
\begin{equation}
\label{compsx2}
\begin{split}
\int_{B_t} \delta_h V_n \cdot \delta_h(\nabla u_n) \dx &= \int_{B_t} (|\nabla (u_n)_h|^{p-2} \nabla (u_n)_h - |\nabla u_n|^{p-2} \nabla u_n) \cdot (\nabla (u_n)_h - \nabla u_n) \dx \\
&\geq c \int_{B_t} (1 + |\nabla (u_n)_h|^2 + |\nabla u_n|^2)^{\frac{p-2}{2}} |\nabla (u_n)_h - \nabla u_n|^2 \dx \\
&= c \int_{B_t} W_{nh} \, |\delta_h (\nabla u_n)|^2 \dx,
\end{split}
\end{equation}
where $c>0$ is sufficiently small while $ W_{nh}:= (1 + |\nabla (u_n)_h|^2 + |\nabla u_n|^2)^{\frac{p-2}{2}} $. H\"older's inequality with exponents $ \frac{2}{p} $ and $ \frac{2}{2-p} $, besides \ref{gradbound}, produce
\begin{equation}
\label{compsx3}
\begin{split}
\| \delta_h (\nabla u_n)\|_{L^p(B_t)}^p &= \int_{B_t} W_{nh}^{\frac{p}{2}} \, |\delta_h (\nabla u_n)|^p \, W_{nh}^{-\frac{p}{2}} \dx \\
&\leq \left( \int_{B_t} W_{nh} \, |\delta_h (\nabla u_n)|^2 \dx \right)^{\frac{p}{2}} \left( \int_{B_t}  W_{nh}^{\frac{p}{p-2}} \dx \right)^{\frac{2-p}{2}} \\
&\leq \left( \int_{B_t} W_{nh} \, |\delta_h (\nabla u_n)|^2 \dx \right)^{\frac{p}{2}} \left(|B_t| + 2 \, \|\nabla u_n\|_{L^p(B_R)}^p\right)^{\frac{2-p}{2}} \\
&\leq C \left( \int_{B_t} W_{nh} \, |\delta_h (\nabla u_n)|^2 \dx \right)^{\frac{p}{2}}.
\end{split}
\end{equation}
Reasoning as in the above case, the conclusion directly follows from \eqref{comptest3}, \eqref{compsx2}, and \eqref{compsx3}.
\end{proof}
\begin{rmk}
\label{alternative}
Lemma \ref{convgrad} can be proved also (in a less direct way) through a result by Boccardo and Murat \cite{BM} which, under the hypotheses of Lemma \ref{convgrad}, ensures that
\begin{equation}
\label{BocMur}
\nabla u_n \to \nabla u \quad \mbox{in} \;\; L^q_{\rm loc}(\Omega) \quad \forall \, q \in (1,p).
\end{equation}
In particular, \eqref{BocMur} implies $ \nabla u_n \to \nabla u $ a.e.~in $ \Omega $. A development of this approach, allowing $q=p$, is contained in \cite[Lemma 2.5 and Remark 3]{GG}. Another way \cite{CM-RatMec,GM} to get convergence of gradient terms is using a differentiability result for the stress field, i.e., the field whose divergence represents the elliptic operator (as $|\nabla u|^{p-2}\nabla u$ for the $p$-Laplacian). In fact,  by Rellich-Kondrachov's theorem \cite[Theorem 9.16]{B},  such a differentiability allows to gain compactness.
\end{rmk}

\begin{thm}
\label{weaksol}
Let \hyperlink{H1}{${\rm H_1}$}--\ref{hypeta} be satisfied. Then problem \eqref{prob} possesses a weak solution $(u,v) \in \D^{1,p}_0(\R^N)\times\D^{1,q}_0(\R^N)$.
\end{thm}
\begin{proof}
The reasoning is patterned after that of \cite[Lemma 4.1]{GMM}, so here we only sketch it. Pick $r,s>1$ such that
\begin{equation}
\label{summability}
\frac{1}{\zeta_1}+\theta_1 < \frac{1}{r'} < 1-\frac{p}{p^*}, \quad \frac{1}{\zeta_2}+\theta_2 < \frac{1}{s'} < 1-\frac{q}{q^*},
\end{equation}
which is possible thanks to \ref{hypa}. Fix $\rho>0$ and define $\eps_n:=\frac{1}{n}$,  $n\in\N$. By \cite[Lemmas 3.5--3.8]{GMM}, for all $n\in\N$ there exists $(u_n,v_n)\in(\D^{1,p}_0(\R^N)\times\D^{1,q}_0(\R^N))\cap C^{1,\alpha}_\loc(\R^N)^2$ solution to
\begin{equation}
\tag{$ {\rm P}^{\eps_n} $}
\left\{
\begin{alignedat}{2}
-\Delta_p u &= f(x,u+\eps_n,v,\nabla u, \nabla v) \; \; &&\mbox{in} \; \; \R^N, \\
-\Delta_q v &= g(x,u,v+\eps_n,\nabla u, \nabla v) \; \; &&\mbox{in} \; \; \R^N, \\
u,v &> 0 \; \; &&\mbox{in} \; \; \R^N,
\end{alignedat}
\right.
\end{equation}
such that the following properties hold true, with appropriate $(u,v)\in\D^{1,p}_0(\R^N)\times\D^{1,q}_0(\R^N)$ and $M,\sigma_{2\rho}>0$:
\begin{equation}
\label{props}
\begin{alignedat}{2}
(u_n,v_n) &\rightharpoonup (u,v) \quad &&\mbox{in} \;\; \D^{1,p}_0(\R^N)\times\D^{1,q}_0(\R^N); \\
(u_n,v_n) &\to (u,v) \quad &&\mbox{in} \;\; W^{1,p}(B_{2\rho})\times W^{1,q}(B_{2\rho}); \\
(u_n,v_n) &\to (u,v) \quad &&\mbox{in} \;\; L^r(B_{2\rho})\times L^s(B_{2\rho}); \\
(\nabla u_n,\nabla v_n) &\to (\nabla u,\nabla v) \quad &&\mbox{a.e. in} \;\; \R^N; \\
\max\left\{\|u_n\|_{L^\infty(\R^N)}, \|v_n\|_{L^\infty(\R^N)}\right\} &\leq M \quad &&\forall\, n\in\N; \\
\min\left\{\inf_{B_{2\rho}} u_n, \inf_{B_{2\rho}} v_n\right\} &\geq \sigma_{2\rho} \quad &&\forall\, n\in\N.
\end{alignedat}
\end{equation}
Hence, \ref{hypf} and \eqref{props} yield, for almost every $x\in B_{2\rho}$,
\begin{equation}
\label{estonball}
\begin{split}
&f(x,u_n(x)+\eps_n,v_n(x),\nabla u_n(x),\nabla v_n(x)) \\
&\leq \hat{m}_1 a_1(x)\left[(u_n(x)+\eps_n)^{\alpha_1}v_n(x)^{\beta_1}+|\nabla u_n(x)|^{\gamma_1}
+|\nabla v_n(x)|^{\delta_1}\right] \\
&\leq \hat{m}_1 a_1(x)\left(\sigma_{2\rho}^{\alpha_1}M^{\beta_1}+|\nabla u_n(x)|^{\gamma_1}+|\nabla v_n(x)|^{\delta_1}\right).
\end{split}
\end{equation}
By \eqref{props} the sequence $\{(\nabla u_n,\nabla v_n)\}$ is bounded in $L^p(B_{2\rho})\times L^q(B_{2\rho})$. Exploiting \ref{hypa}, \eqref{summability}, and \eqref{estonball} we thus see that
\begin{equation*}
\{f(\cdot,u_n+\eps_n,v_n,\nabla u_n,\nabla v_n)\} \quad \mbox{is bounded in} \;\; L^{r'}(B_{2\rho}).
\end{equation*}
Accordingly, Lemma \ref{convgrad}, with $\Omega:=B_{2\rho}$, besides \eqref{props}, produces $\nabla u_n\to\nabla u$ in $L^p(B_\rho)$. Now the proof goes on exactly as in \cite[Lemma 4.1]{GMM}, ensuring that $(u,v)$ is a distributional solution to \eqref{prob}. The conclusion is achieved by invoking \cite[Lemma 4.2]{GMM}, which shows that any distributional solution to \eqref{prob} is actually a weak one.
\end{proof}

\begin{rmk}
\label{strongsol}
An advantage of using differentiability results for the stress field (see Remark \ref{alternative}) in this context is the possibility to obtain strong solutions to \eqref{prob}, as done in \cite[Lemma 4.3]{GMM}: indeed, otherwise we do not know how to give a pointwise (a.e.)~sense to the $p$-Laplacian operator, seen as the divergence of the stress field $|\nabla u|^{p-2}\nabla u$. This issue is linked to a well-known conjecture for \eqref{mingioneprob}, which can be stated as
\begin{equation*}
f\in L^{r}_\loc(\Omega) \quad \stackrel{?}{\Leftrightarrow} \quad |\nabla u|^{p-2}\nabla u \in W^{1,r}_\loc(\Omega).
\end{equation*}
For a discussion about this conjecture, see \cite[Section 1]{GM}.
\end{rmk}

\section*{Acknowledgments}

The authors are members of \textit{Gruppo Nazionale per l'Analisi Matematica, la Probabilità e le loro Applicazioni} (GNAMPA) of the \textit{Istituto Nazionale di Alta Matematica} (INdAM).\\
They were supported by the following research projects: 1) PRIN 2017 `Nonlinear Differential Problems via Variational, Topological and Set-valued Methods' (Grant no. 2017AYM8XW) of MIUR; 2) `MO.S.A.I.C.' PRA 2020--2022 `PIACERI' Linea 2 (S.A. Marano) and Linea 3 (U. Guarnotta) of the University of Catania. U. Guarnotta also acknowledges the support of GNAMPA-INdAM Project CUP\_E55F22000270001.

\noindent
\textbf{Conflict of interest statement.} On behalf of all authors, the corresponding author states that there is no conflict of interest.

\begin{small}

\end{small}

\end{document}